\documentclass[twoside,11pt,leqno]{article}
\newcommand{\oR}{{\mathbb R}}

\newcommand{\oN}{{\mathbb N}}

\newcommand{\EE}{{\mathbb E}}
\newcommand{\PP}{{\mathbb P}}

\newcommand{\var}{{\rm{Var}}}
\newcommand{\mse}{{\rm{mse}}}
\newcommand{\bias}{{\rm{bias}}}

\newtheorem{res}{Lemma}
\newtheorem{cor}{Corollary}
\newtheorem{thm}{Theorem}
\newtheorem{prop}{Proposition}

\newenvironment{proof}[1]{\noindent{\bf Proof of {#1}:}}{\hfill $\square$ \\}
\usepackage{epsf,amssymb,amsmath}

\begin{document}
\bibliographystyle{plain}
\thispagestyle{empty}
\begin{center}
{\Large\bf Infill asymptotics and bandwidth selection
for kernel estimators of spatial intensity functions}\\[.4in]

\noindent
{\large M.N.M. van Lieshout}\\[.1in]
\noindent
{\em CWI, P.O.~Box 94079, NL-1090 GB  Amsterdam\\
University of Twente, P.O.~Box 217, NL-7500 AE Enschede\\
The Netherlands}\\[.1in]
\end{center}

\begin{verse}
{\footnotesize
\noindent
{\bf Abstract}\\
\noindent
We investigate the asymptotic mean squared error of kernel
estimators of the intensity function of a spatial point process.
We show that when $n$ independent copies of a point process in
$\oR^d$ are superposed, the optimal bandwidth $h_n$ is
of the order $n^{-1/(d+4)}$ under appropriate smoothness conditions on
the kernel and true intensity function. We apply the Abramson principle
to define adaptive kernel estimators and show that asymptotically
the optimal adaptive bandwidth is of the order $n^{-1/(d+8)}$ under
appropriate smoothness conditions. \\[0.1in]

\noindent
{\em Keywords \& Phrases:}
Adaptive kernel estimator; Bandwidth; Infill asymptotics;
Intensity function; Kernel estimator;
Mean squared error; Point process. \\[0.1in]
}
\end{verse}

\section{Introduction}

Often the first step in the analysis of a spatial point pattern is to estimate
its intensity function. Various non-parametric estimators are available to do so.
Some techniques are based on local neighbourhoods of a point, expressed for
example by its nearest neighbours \cite{Gran98}, its Voronoi \cite{Ord78} or
Delaunay tessellation \cite{Scha07,SchaWeyg00}. By far the most popular technique,
however, is kernel smoothing \cite{Digg85}. Specifically, let $\Phi$ be a point 
process that is observed in a bounded open subset $\emptyset \neq W$ of $\oR^d$
and assume that its first order moment measure exists as a $\sigma$-finite 
Borel measure and is absolutely continuous with respect to Lebesgue measure with 
a Radon--Nikodym derivative $\lambda : \oR^d \to [0,\infty)$  known as its 
intensity function.  A kernel estimator of $\lambda$ based on $\Phi \cap W$ 
takes the form
\begin{equation}
\label{e:kernel}
\widehat{ \lambda(x_0;h)}
=
\widehat{ \lambda(x_0; h, \Phi, W)}
=
\frac{1}{h^d}
 \sum_{y\in\Phi\cap W} 
\kappa\left( \frac{x_0 - y}{h}\right),
\quad x_0 \in W.
\end{equation}
The function $\kappa: \oR^d \rightarrow [0,\infty)$ is supposed to
be kernel, that is, a $d$-dimensional symmetric probability density
function \cite[p.~13]{Silv86}. The choice of bandwidth $h>0$ determines
the amount of smoothing.  In principle, the support of 
$\kappa((x_0 - y)/ h)$ as a function of $y$ could overlap the
complement of $W$. Therefore, various edge corrections have been
proposed  \cite{BermDigg89,Lies12}. In the sequel, though, we will be
concerned with very small bandwidths, so this aspect may be ignored.

The aim of this paper is to derive asymptotic expansions for the
bias and variance of (\ref{e:kernel}) in terms of the bandwidth.
This problem is well known when dealing with probability density 
functions. Indeed, there exists a vast literature, for example the 
textbooks \cite{BowmAzza97,Silv86,WandJone94} and the references therein.
In a spatial context, bandwidth selection is dominated by ad hoc
\cite{BermDigg89} and non-parametric methods \cite{CronLies18}. The
first rigorous study into bandwidth selection to the best of our
knowledge is that by Lo \cite{Lo17} who studies infill asymptotics
for spatial patterns consisting of independent and identically
distributed points. Our goal is to extend his approach to point
processes that may exhibit interactions between their points
and to investigate adaptive versions thereof.

The plan of this paper is as follows. In Section~\ref{S:local}
we focus on the regime in which $n$ independent copies of
the same point process are superposed and the bandwidth $h_n$ tends
to zero as $n$ tends to infinity but does not depend on the points
of the pattern. We derive Taylor expansions and deduce the asymptotically
optimal bandwidth. Intuitively, however, one feels that in sparse
regions more smoothing is necessary then in regions that are rich
in points. Indeed, in the context of estimating a probability density
function, Abramson \cite{Abra82a} proposed to scale the bandwidth in
proportion to the square root of the density. Analogously, in
Section~\ref{S:adaptive} we let $h_n$ decrease in proportion to
the square root of the intensity function and show that by doing so
the bias can be reduced. For the sake of readability, all proofs are 
deferred to Section~\ref{S:proofs}. 

\section{Infill asymptotics}
\label{S:local}

Let $\Phi_1, \Phi_2, \ldots$ be independent and identically distributed
point processes for which the first order moment measure exists, is locally
finite and admits an  intensity function $\lambda: W \to [0,\infty)$. 
For $n\in \oN$, let 
\[
Y_n = \bigcup_{i=1}^n \Phi_i
\]
denote the union.  Upon taking the limit for $n\to \infty$, one obtains an 
asymptotic regime known as `infill asymptotics' \cite{Ripl88}.
Since the $\Phi_i$ are independent, the intensity function of
$Y_n$ is $n\lambda( \cdot )$. Therefore $\lambda(x_0)$, $x_0 \in W$, may be estimated 
by 
\begin{equation} \label{e:lambdahat}
\widehat{\lambda(x_0)} :=  \frac{\widehat{ \lambda(x_0; h, Y_n, W)} }{n}
  = \frac{1}{n} \sum_{i=1}^n
  \widehat{ \lambda( x_0; h, \Phi_i, W ) }.
\end{equation}

\begin{res} \label{L:moments}
Let $\Phi$ be a point process observed in a bounded open subset
$\emptyset \neq W \subset \oR^d$ whose factorial moment measures
exist up to second order and are absolutely continuous with intensity
function  $\lambda$ and second order product densities $\rho^{(2)}$.
Let $\kappa$ be a kernel. Then the first two moments of (\ref{e:kernel}) are 
\[
\EE \left[ \widehat{\lambda(x_0; h, \Phi, W)} \right] =
\frac{1}{h^d} \int_{ W} 
\kappa\left( \frac{x_0 - u}{h} \right) \lambda(u) du
\]
and
\begin{eqnarray*}
\EE \left[ \left( \widehat{\lambda(x_0; h, \Phi, W)} \right)^2 \right]  
&  = &
\frac{1}{h^{2d}} \int_{W} \int_{W}
   \kappa\left( \frac{x_0 - u}{h} \right)
   \kappa\left( \frac{x_0 - v}{h} \right)
 \rho^{(2)}(u,v) du dv
\\
& + &
\frac{1}{h^{2d}}
\int_{W}  \kappa\left( \frac{x_0 - u}{h} \right)^2  \lambda(u) du.
\end{eqnarray*}
\end{res}

The proof follows directly from the definition of product densities,
see for example \cite[Section~4.3.3]{SKM}.
Provided $\lambda(\cdot) > 0$, the variance of $\widehat{\lambda(x_0)}$
can expressed in terms of the pair correlation function $g: W\times W\to \oR$
defined by $g(u,v) = \rho^{(2)}(u,v)/ (\lambda(u) \lambda(v))$ as
\[
\frac{1}{h^{2d}} \left[
\int_{W \times W}
   \kappa\left( \frac{x_0 - u}{h} \right)
   \kappa\left( \frac{x_0 - v}{h} \right)
   ( g(u,v) - 1 ) \lambda(u) \lambda(v) du dv
 +
 \int_{W}    \kappa\left( \frac{x_0 - u}{h} \right)^2   \lambda(u) du
 \right].
\]
For Poisson processes, the first integral vanishes as $g\equiv 1$.

In this paper, we will restrict ourselves to kernels that belong to the Beta class
\begin{equation}
\label{e:beta}
\kappa^\gamma(x) = 
\frac{\Gamma \left(d/2 + \gamma + 1\right)}{
\pi^{d/2} \Gamma \left(\gamma+1\right)}
(1 - x^T x)^{\gamma} \, 1\{ x \in b(0, 1) \},
\quad x\in\oR^d,
\end{equation}
for $\gamma \geq 0$. Here $b(0,1)$ is the closed unit ball in $\oR^d$ centred at the
origin. The normalising constant will be abbreviated by
\begin{equation}
\label{e:cBeta}
c(d,\gamma) = \int_{b(0,1)} (1 - x^Tx)^\gamma dx =
\frac{\pi^{d/2} \Gamma(\gamma+1) }{ \Gamma(d/2 + \gamma + 1)}, \quad d\in\oN, \gamma \geq 0.
\end{equation}
Note that Beta kernels are supported on the compact unit ball and that their smoothness
is governed by the parameter $\gamma$. Indeed, the box kernel defined by $\gamma=0$
is constant and therefore continuous on the interior of the unit ball; the Epanechnikov 
kernel corresponding to the choice $\gamma=1$ is Lipschitz continuous. For $\gamma > k$ 
the function $\kappa^\gamma$ is $k$ times continuously differentiable on $\oR^d$.

The following Lemma collects further basic properties of the Beta kernels.
The proof can be found in Section~\ref{S:proofsB}.

\begin{res} \label{L:Beta}
For the Beta kernels $\kappa^\gamma$, $\gamma \geq 0$, defined in
equation (\ref{e:beta}), the  integrals
  \[
  \int_{\oR} x_i \kappa^\gamma(x) dx_i = 0 =
  \int_{b(0,1)}  x_i x_j \kappa^\gamma(x) dx_1 \cdots dx_d
\]
vanish for all $i, j \in \{ 1, \dots, d\}$ such that $i\neq j$.
Furthermore
\[
  Q(d, \gamma) := \int_{\oR^d} \kappa^\gamma(x)^2 dx = \frac{
    c(d,2\gamma)}{c(d,\gamma)^2}
\]
is finite and so are, for all $i=1, \dots, d$, 
\[
  V(d,\gamma) :=  \int_{-\infty}^\infty \cdots \int_{-\infty}^\infty 
   x_i^2 \kappa^\gamma(x) dx_1 \cdots dx_d = \frac{1}{d+2\gamma+2}, 
\]
\[
  V_4(d, \gamma)  :=  \int_{-\infty}^\infty \cdots \int_{-\infty}^\infty 
   x_i^4 \kappa^\gamma(x) dx_1 \cdots dx_d = \frac{3}{(d+2\gamma+2)(d+2\gamma+4)} 
\]
as well as, for $d\geq 2$ and $i\neq j \in \{ 1, \dots, d \}$,
\[
  V_2(d, \gamma) := \int_{-\infty}^\infty \cdots \int_{-\infty}^\infty 
   x_i^2 x_j^2 \kappa^\gamma(x) dx_1 \cdots dx_d
  = \frac{1}{(d+2\gamma+2) (d+2\gamma+4)}.
\]
Their values do not depend on the particular choices of $i$ and $j$.
\end{res}

For the important special case $d=2$, 
\[
  Q(2, \gamma) = \frac{(\gamma+1)^2}{(2\gamma + 1) \pi}.
\]

Lemma~\ref{L:moments} can be used to derive the mean squared error of
(\ref{e:lambdahat}). Its proof can be found in Section~\ref{S:proofs-nA}.

\begin{prop} \label{P:mse}
Let $\Phi_1, \Phi_2, \ldots$ be independent and identically distributed
point processes observed in a bounded open subset $\emptyset \neq W \subset \oR^d$.
Assume that their factorial moment measures exist up to second order
and are absolutely continuous with strictly positive intensity function
$\lambda: W \to (0,\infty)$ and second order product densities $\rho^{(2)}$.
Write
\(
Y_n = \bigcup_{i=1}^n \Phi_i
\)
for the union, $n\in \oN$, and let $\kappa^\gamma(x)$ be a Beta kernel
(\ref{e:beta}) with $\gamma \geq 0$. Then the mean squared error of
(\ref{e:lambdahat}) is given by
\begin{eqnarray*}
\mse \widehat{\lambda(x_0)}  & = &  \left(
  \frac{1}{h^{d}}   \int_{b(x_0, h) \cap W} 
     \kappa^\gamma\left( \frac{x_0 - u}{h} \right)
     \lambda(u) du - \lambda(x_0)  \right)^2 \\
& + &
\frac{1}{n h^{2d}}
\int \int_{(b(x_0, h) \cap W)^{2}} 
   \kappa^\gamma\left( \frac{x_0 - u}{h} \right)
   \kappa^\gamma\left( \frac{x_0 - v}{h} \right)
   ( g(u,v) - 1 ) \lambda(u) \lambda(v) du dv \\
& + &
\frac{1}{nh^{2d}}
\int_{b(x_0, h) \cap W}    \kappa^\gamma\left( \frac{x_0 - u}{h} \right)^2
   \lambda(u) du.
\end{eqnarray*}
\end{prop}

The first term in the above expression is the squared bias. It depends on 
$\lambda$ and the bandwidth $h$ but not on $n$. The remaining terms come from 
the variance and depend on $\lambda$, on $g$, on $h$ and on $n$. 

Our aim in the remainder of this section is to derive an asymptotic expansion 
of the mean squared error for bandwidths $h_n$ that depend on $n$ in such a 
way that $h_n\to 0$ as $n\to\infty$. In order to achieve this, first recall
some basic facts from analysis. Let $E$ be an open subset of $\oR^n$ and denote 
by ${\cal{C}}^k(E)$ the class of functions $f: E \to \oR^m$ for which all 
$k^{\rm{th}}$ order partial derivatives $D_{j_1 \cdots j_k}f$ exist and are 
continuous on $E$. For such functions the order of taking partial derivatives 
may be interchanged and the Taylor theorem states that if $x \in E$ and 
$x+th\in E$ for all $0\leq t\leq 1$, then a $\theta \in (0,1)$ can be found 
such that 
\begin{equation}
\label{e:Taylor}
f(x+h) - f(x) = \sum_{r=1}^{k-1} \frac{1}{r!} D^rf(x) (h^{(r)}) +
\frac{1}{k!} D^kf(x+\theta h)(h^{(k)}),
\end{equation}
where $h^{(r)}$ is the $r$-tuple $(h, \dots, h)$ and
\[
D^r f(x) (h^{(r)}) := \sum_{j_1, \dots, j_r = 1}^n 
h_{j_1} \cdots h_{j_r} D_{j_1 \cdots j_r}f(x)
\]
for $h = (h_1, \dots, h_n) \in \oR^n$.

We are now ready to state the main result of this section, generalising
\cite[Theorem~2]{Lo17} for the union of independent random points.
The proof can be found in Section~\ref{S:proofs-nA}.

\begin{thm} \label{P:biasVar}
Let $\Phi_1, \Phi_2, \ldots$ be i.i.d.\ point processes observed in a
bounded open subset $\emptyset \neq W \subset \oR^d$ with
well-defined intensity function $\lambda$ and pair correlation function $g$.
Suppose that $g: W\times W \to \oR$ is bounded and that
$\lambda: W \to (0, \infty)$ is twice continuously differentiable 
with second order partial derivatives
$\lambda_{ij} = D_{ij}\lambda$, $i, j = 1, \dots, d$, that are
H\H{o}lder continuous with index $\alpha > 0$ on $W$, that is,
there exists some $C>0$ such that for all $i, j = 1, \dots, d$:
\[
| \lambda_{ij}(x) - \lambda_{ij}(y) | \leq C || x - y ||^\alpha,
\quad x, y \in W.
\]

Consider the estimator $\widehat \lambda$ based on the unions
$Y_n = \cup_{i\leq n} \Phi_i$, $n\in \oN$, and Beta kernel
$\kappa^\gamma$, $\gamma \geq 0$, with bandwidth $h_n$ chosen in such a
way that, as $n\to \infty$,  $h_n \to 0$ and $n h_n^{d} \to \infty$. 
Then, for $x_0 \in W$, as $n\to \infty$,
\begin{enumerate}
\item $\bias \widehat{\lambda(x_0)} =
 h_n^2 \frac{\sum_{i=1}^d \lambda_{ii}(x_0)}{2 (d + 2\gamma + 2) } +
 O(h_n^{2+\alpha})$.
\item $\var \widehat{\lambda(x_0)}= \frac{\lambda(x_0) Q(d,\gamma)}{nh_n^d}
+  O\left(\frac{1}{nh_n^{d-1}}\right)$.
\end{enumerate}
\end{thm}

The bias depends on the second order partial derivatives of the
unknown intensity function and on the smoothness parameter $\alpha$.
The smoothness of the kernel, measured by $\gamma$, also plays a role.
The leading term of the variance depends on $\lambda(x_0)$ and on the 
smoothness of the kernel.

Theorem~\ref{P:biasVar} readily yields the asymptotically
optimal bandwidth, cf.\ Section~\ref{S:proofs-nA}.

\begin{cor} \label{C:optimal}
Consider the setting of Theorem~\ref{P:biasVar}. Then 
\[
  \mse \widehat{\lambda(x_0)}  =
  h_n^4 \frac{V(d,\gamma)^2}{4} \left( \sum_{i=1}^d \lambda_{ii}(x_0) \right)^2 +
   \frac{ \lambda(x_0) Q(d, \gamma) } {nh_n^d} +
 +   O\left(h_n^{4+\alpha} \right) + O\left(\frac{1}{n h_n^{d-1}} \right).
\]
The asymptotic mean squared error is optimised at
\[
  h_n^*(x_0) = \frac{1}{n^{1/(d+4)}} \left(
    \frac{d \lambda(x_0) Q(d,\gamma)}{ V(d,\gamma)^2
      \left(  \sum_{i=1}^d \lambda_{ii}(x_0)\right)^2}
  \right)^{1/(d+4)}.
\]
\end{cor}

In words, $h_n^*(x_0)$ is of the order $n^{-1/(d+4)}$. Clearly $h_n^*(x_0)$ tends
to zero as $n\to \infty$. Moreover, $n (h_n^*)^d$ is of the order $n$ to the $1 - d/(d+4)$
and therefore tends to infinity with $n$. For the special case $d=2$,
\[
  h_n^*(x_0) = \frac{1}{n^{1/6}} \left(
   \frac{ 8 \lambda(x_0) (\gamma + 1)^2 (\gamma+2)^2}{(2\gamma+1) \pi
   (\lambda_{11}(x_0) + \lambda_{22}(x_0))^2}
    \right)^{1/6}.
\]

The following Proposition generalises \cite[Proposition~5]{Lo17}.
Its proof can be found in Section~\ref{S:proofs-nA}.

\begin{prop}  \label{P:bigO}
Let $\Phi_1, \Phi_2, \ldots$ be i.i.d.\ point processes observed in 
a bounded  open subset $\emptyset \neq W \subset \oR^d$ with
well-defined intensity function $\lambda$ and pair correlation 
function $g$. Suppose that $g:W\times W \to \oR$ is bounded and that
$\lambda: W \to (0, \infty)$ is twice continuously differentiable
with second order partial derivatives
$\lambda_{ij} = D_{ij}\lambda$, $i, j = 1, \dots, d$, that are
H\H{o}lder continuous with index $\alpha > 0$ on $W$. 
Consider $\widehat \lambda$ based on the unions
$Y_n = \cup_{i\leq n} \Phi_i$, $n\in \oN$, and Beta kernel
$\kappa^\gamma$, $\gamma \geq 0$, with bandwidth $h_n$ chosen in
such a way that as $n\to \infty$, $h_n \to 0$ and
$n h_n^{d} \to \infty$. Then, for $x_0\in W$, as $n\to \infty$,
\[
  \widehat {\lambda(x_0)} 
  = \lambda(x_0) + h_n^2 \frac{\sum_{i=1}^d \lambda_{ii}(x_0)}{2(d+2\gamma + 2)}
 +  O(h_n^{2+\alpha})
  + \sqrt{ \lambda(x_0) Q(d,\gamma) } O_P\left( n^{-1/2}h_n^{-d/2} \right).
\]
\end{prop}

\section{Adaptive infill asymptotics}
\label{S:adaptive}

Up to now, estimators based on (\ref{e:kernel}) were considered
in which the same bandwidth $h$ was applied at every point 
$y \in \Phi \cap W$. However, at least intuitively, it seems
clear that the bandwidth should be smaller in regions with
many points, larger when points are scarce. This suggests
that $h = h(y)$ should be decreasing in $\lambda(y)$. 

Motivated by similar considerations in the context of density 
estimation, Abramson \cite{Abra82a} suggested to consider 
point-dependent bandwidths of the form
\(
h(y) =  h / c(y)
\)
for $c(y)$ equal to the square root of the probability density
function. He found that a significant reduction in bias could be
obtained by the use of such adaptive bandwidths. Our aim in this
section is to show that a similar result holds for spatial 
intensity function estimation.

Define an estimator 
\[
{\tilde \lambda(x_0)}   
 = \frac{1}{n} \sum_{i=1}^n
  \widehat{ \tilde \lambda( x_0; h, \Phi_i, W ) }
\]
of $\lambda(x_0)$, $x_0 \in W$, that is the average of
data-adaptive estimators 
\begin{equation} \label{e:Abramson}
  { \tilde \lambda( x_0; h, \Phi_i, W ) } =
  \sum_{y\in \Phi_i}  \frac{c(y)^d}{h^d} \kappa\left(
    \frac{x_0-y}{h} c(y) \right).
\end{equation}
As in Section~\ref{S:local}, $\kappa$ is a kernel and the $\Phi_i$,
$i=1, \dots, n$, are independent and identically distributed point processes
on $\oR^d$ observed in a bounded non-empty open subset $W$ for which the
first order moment measure exists and admits an intensity function 
$\lambda: W \to [0,\infty)$; $c: W \to (0,\infty)$ is assumed to be a 
measurable positive-valued weight function on $W$.

The next result summarises the first two moments.

\begin{res} \label{L:momentsA}
Let $\Phi$ be a point process observed in a bounded open subset
$\emptyset \neq W \subset \oR^d$, whose factorial moment measures
exist up to second   order and are absolutely continuous with intensity
function $\lambda$ and second order product densities $\rho^{(2)}$. Let
$\kappa$ be a kernel. Then the first two moments of (\ref{e:Abramson}) are
\[
\EE \tilde \lambda(x_0) =
\frac{1}{h^d} \int_{W} c(u)^d
\kappa\left( \frac{x_0 - u}{h} c(u) \right) \lambda(u) du
\]
and
\begin{eqnarray*}
  \EE \left[ \left( {\tilde \lambda(x_0;h,\Phi_1, W)} \right)^2 \right]
  & = &
  \frac{1}{h^{2d}} \int_{W^2} c(u)^d c(v)^d
    \kappa\left( \frac{x_0-u}{h} c(u) \right)
    \kappa\left( \frac{x_0-v}{h} c(v) \right) \rho^{2}(u,v) du dv
  \\
  & + &        \frac{1}{ h^{2d}} \int_W c(u)^{2d}
                  \kappa\left( \frac{x_0-u}{h} c(u) \right)^2 \lambda(u) du.
\end{eqnarray*}
\end{res}

The proof follows directly from the definition of product densities,
see for example \cite[Section~4.3.3]{SKM}.
For the special case $c(u) \equiv 1$, we retrieve Lemma~\ref{L:moments}.

Provided $\lambda(\cdot) > 0$, the variance of $\tilde \lambda(x_0)$,
the average of the $\tilde \lambda(x_0; h, \Phi_i, W)$,
can be expressed in terms of the pair correlation function as
\[
  \var {\tilde \lambda(x_0)}   =
  \frac{1}{nh^{2d}} \int_W c(u)^{2d}
      \kappa\left( \frac{x_0 - u}{h} c(u) \right)^2 \lambda(u) du +
\]
\begin{equation}\label{e:varA}
\frac{1}{n h^{2d}}
\int_W \int_W c(u)^d c(v)^d
   \kappa\left( \frac{x_0 - u}{h} c(u) \right)
   \kappa\left( \frac{x_0 - v}{h} c(v) \right)
   ( g(u,v) - 1 ) \lambda(u) \lambda(v) du dv.
\end{equation}

We are now ready to state the first main result of this section in analogy to
\cite[Theorem, p.~1218]{Abra82a}. The proof can be found in Section~\ref{S:proofs-A}.

\begin{thm}  \label{P:biasVarA}
Let $\Phi_1, \Phi_2, \ldots$ be i.i.d.\ point processes observed in a
bounded open subset $\emptyset \neq W \subset \oR^d$ with well-defined
intensity function $\lambda$ and pair correlation function $g$. Suppose that 
$g: W\times W \to \oR$ is bounded and that 
$\lambda: W \to (\underline \lambda, \bar \lambda)$ is bounded, bounded away 
from zero and twice continuously differentiable on $W$ with bounded second 
order partial derivatives $\lambda_{ij} = D_{ij}\lambda$, $i, j = 1, \dots, d$. 

Consider the estimator $\tilde \lambda$ with
\[
  c(x) =  \sqrt{ \frac{ \lambda(x) }{\lambda(x_0) } }
\]
based on the unions $Y_n = \cup_{i\leq n} \Phi_i$, $n\in \oN$, and Beta kernel
$\kappa^\gamma$, $\gamma > 2$, with bandwidth $h_n$ chosen in such a way that,
as $n\to \infty$, $h_n \to 0$ and $n h_n^{d} \to \infty$.
Then, for $x_0 \in W$, as $n\to \infty$,
\begin{enumerate}
\item $\bias \tilde \lambda(x_0) = o(h_n^2)$.
\item $\var \tilde \lambda(x_0) =  \frac{\lambda(x_0) Q(d,\gamma)}{nh_n^d}
  + O\left( \frac{1}{nh_n^{d-1}}\right)$.
\end{enumerate}
\end{thm}

In comparison with Theorem~\ref{P:biasVar}, the variance is the same as
that for a non-adaptive bandwidth. The bias term on the other hand is of
a smaller order. Note that, since the leading bias term is not specified,
Theorem~\ref{P:biasVarA} cannot be used to calculate an asymptotically
optimal bandwidth. To remedy this, stronger smoothness assumptions seem
needed.

\begin{thm}  \label{P:biasVarHall}
Let $\Phi_1, \Phi_2, \ldots$ be i.i.d.\ point processes observed in a
bounded open subset $\emptyset \neq W \subset \oR^d$ with well-defined
intensity function $\lambda$ and pair correlation function $g$. Suppose that
$g: W\times W \to \oR$ is bounded and that
$\lambda: W \to (\underline \lambda, \bar \lambda)$ is bounded, bounded away 
from zero and five times continuously differentiable on $W$ with bounded 
partial derivatives.

Consider the estimator $\tilde \lambda$ with
\[
  c(x) =  \sqrt{ \frac{ \lambda(x) }{\lambda(x_0) } }
\]
based on the unions $Y_n = \cup_{i\leq n} \Phi_i$, $n\in \oN$, and Beta kernel
$\kappa^\gamma$, $\gamma > 5$, with bandwidth $h_n$ chosen in such a way that,
as $n\to \infty$, $h_n \to 0$ and $n h_n^{d} \to \infty$.
Then, for $x_0 \in W$, as $n\to \infty$,
\begin{enumerate}
\item $\bias \tilde \lambda(x_0) = \lambda(x_0) h_n^4 \int_{\oR^d} A(u; x_0)du + o(h_n^4)$,
where 
\begin{eqnarray*}
A(u; x_0) & = &
  \frac{Dg_u(1)}{24} D^4 c(x_0)(u,u,u,u) + \frac{ D^4g_u(1)}{24} ( Dc(x_0)u)^4\\
& + &
  \frac{D^2g_u(1)}{2} \left\{ \frac{1}{3} D c(x_0) u \,  D^3 c(x_0) (u,u,u)
  + \frac{1}{4} (D^2c(x_0)(u,u))^2 \right\} \\
& + &
  \frac{D^3g_u(1)}{4}  (D c(x_0)u)^2 D^2c(x_0)(u,u) 
\end{eqnarray*}
and $g_u(v) = v^{d+2} \kappa^\gamma(vu)$.
\item $\var \tilde \lambda(x_0) =  \frac{\lambda(x_0) Q(d,\gamma)}{nh_n^d}
  + O\left( \frac{1}{nh_n^{d-1}}\right)$.
\end{enumerate}
\end{thm}

For the important special cases $d=1, 2$, the expression for $A(u; x_0)$ may be
simplified. All the proofs are given in Section~\ref{S:proofs-A}.

\begin{prop} \label{P:dim1}
Consider the framework of Theorem~\ref{P:biasVarHall} in one dimension
$d=1$.
Then the coefficient of $h_n^4$ in the expansion of
$\bias \tilde \lambda(x_0)$ is
\[
\frac{ \lambda(x_0)V_4(1,\gamma)}{24}
 \left[
    - \frac{\lambda^{(iv)}(x_0)}{\lambda(x_0)} +
    8 \frac{\lambda^{\prime\prime\prime}(x_0) \lambda^\prime(x_0)}{\lambda(x_0)^2}
    + 6 \frac{(\lambda^{\prime\prime}(x_0))^2}{\lambda(x_0)^2}
   -36 \frac{\lambda^{\prime\prime}(x_0) (\lambda^\prime(x_0))^2}{\lambda(x_0)^3}
   + 24 \frac{ (\lambda^\prime(x_0))^4}{\lambda(x_0)^4}
 \right]
\]
where $V_4(1,\gamma) = 3 / ( (3+2\gamma)(5+2\gamma) )$ and the superscript $(iv)$
indicates the fourth order derivative.
\end{prop}

\begin{prop} \label{P:dim2}
Consider the framework of Theorem~\ref{P:biasVarHall} in two dimensions
$d=2$. Then the coefficient of $h_n^4$ in the expansion of
$\bias \tilde \lambda(x_0)$ is
\[
  \lambda(x_0) \left\{ V_4(2,\gamma) C_4 +
   V_2(2,\gamma) C_2 \right\},
\]
with $V_4(2,\gamma) = 3/( (4 + 2\gamma) (6 + 2\gamma) )$,
$V_2(2,\gamma) = 1/( (4 + 2\gamma) (6 + 2\gamma) )$ and constants
\[
  C_4  =
  \sum_{i=1}^2 \left[ \frac{-1}{12} D_{iiii}c(x_0) +
            D_ic(x_0) D_{iii}c(x_0) + \frac{3}{4} ( D_{ii}c(x_0) )^2
            - 6 (D_i c(x_0))^2 D_{ii} c(x_0) + 5 ( D_i c(x_0) )^4
  \right]
\]
and
\begin{eqnarray*}
C_2  & = & 30 (D_1 c(x_0))^2 (D_2 c(x_0))^2
   - 6  (D_1 c(x_0))^2 D_{22}c(x_0) - 6 (D_2 c(x_0))^2 D_{11}c(x_0) \\
   & - & 24 D_1 c(x_0) D_2 c(x_0) D_{12} c(x_0)
        + 3 D_1 c(x_0) D_{122}c(x_0) + 3 D_2 c(x_0) D_{112}c(x_0) \\
 &  + & \frac{3}{2}  D_{11} c(x_0) D_{22}c(x_0) + 3 (D_{12} c(x_0))^2 
    - \frac{1}{2} D_{1122}c(x_0).
\end{eqnarray*}
\end{prop}

Theorem~\ref{P:biasVarHall} immediately yields the asymptotically optimal
bandwidth, which should be compared with that in Corollary~\ref{C:optimal}.

\begin{cor} \label{C:optimalHall}
Consider the setting of Theorem~\ref{P:biasVarHall}. Then
\[
  \mse \tilde \lambda(x_0) =
  \lambda(x_0)^2 \left( \int_{\oR^d} A(u; x_0) du \right)^2 h_n^8 +
  \frac{\lambda(x_0) Q(d,\gamma)}{nh_n^d} + o(h_n^8) +
  O \left( \frac{1}{nh_n^{d-1}} \right). 
\]
The asymptotic mean squared error is optimised at
\[
  h_n^*(x_0) = \frac{1}{n^{1/(d+8)}}
  \left( \frac{ d Q(d,\gamma)}{
      8 \lambda(x_0) \left( \int_{\oR^d} A(u; x_0) du \right)^2} \right)^{1/(d+8)}.
\]
\end{cor}

The optimal bandwidth $h_n^*(x_0)$ and the weights $( \lambda(x) / \lambda(x_0) )^{1/2}$ 
depend on the unknown intensity function. In practice, a non-parametric pilot estimator 
(for example the one proposed in \cite{CronLies18}) would be plugged in. 

To conclude this section, we present the analogue of Proposition~\ref{P:bigO}.
The proof can be found in Section~\ref{S:proofs-A}.

\begin{prop}
  \label{P:bigOHall}
Let $\Phi_1, \Phi_2, \ldots$ be i.i.d.\ point processes observed in a
bounded open subset $\emptyset \neq W \subset \oR^d$ with well-defined
intensity function $\lambda$ and pair correlation function $g$.
Suppose that $g:W\times W \to \oR$ is bounded and that
$\lambda: W \to (\underline \lambda, \overline \lambda)$ is
bounded, bounded away from zero and five times continuously differentiable
on $W$ with bounded partial derivatives.
Consider $\tilde \lambda$ with $c(x) = (\lambda(x)/\lambda(x_0))^{1/2}$
based on the unions
$Y_n = \cup_{i\leq n} \Phi_i$, $n\in \oN$, and Beta kernel
$\kappa^\gamma$, $\gamma >5$, with bandwidth $h_n$ chosen in
such a way that as $n\to \infty$, $h_n \to 0$ and
$n h_n^{d} \to \infty$. Then, for $x_0\in W$, as $n\to \infty$,
\[
 {\tilde \lambda(x_0)} 
  = \lambda(x_0) + h_n^4 \lambda(x_0) \int_{\oR^d} A(u; x_0) du + o(h_n^{4})
  + \sqrt{ \lambda(x_0) Q(d,\gamma) } O_P\left( n^{-1/2}h_n^{-d/2} \right)
\]
where $A(u; x_0)$ is as defined in Theorem~\ref{P:biasVarHall}.
\end{prop}

\section{Proofs and technicalities} \label{S:proofs}

\subsection{Auxiliary lemmas for the Beta kernel}
\label{S:proofsB}

\begin{proof}{Lemma~\ref{L:Beta}}
The first two claims follow from the symmetry of the Beta kernel.
Furthermore
\[
  Q(d, \gamma) = \int_{\oR^d} \kappa^\gamma(x)^2 dx = \frac{1}{c(d,\gamma)^2}
  \int_{b(0,1)} ( 1 - ||x ||^2)^{2\gamma} dx = 
  \frac{c(d, 2\gamma)}{c(d, \gamma)^2}.
\]

Due to the symmetry of the Beta kernel it is clear that the definitions
of $V(d,\gamma)$, $V_4(d,\gamma)$ and $V_2(d,\gamma)$ do not depend
on the choices of $i$ and $j$. First consider the case $d=1$.
By the symmetry of $\kappa^\gamma$ and a change
of variables $v = x^2$, $dx = dv / (2 \sqrt v)$, it follows that
\[
  V(1, \gamma) = \int_{-\infty}^\infty x^2 \kappa^\gamma(x) dx =
 \frac{2}{c(1,\gamma)} \int_{0}^1 v (1 - v)^\gamma  \frac{1}{2 v^{1/2}} dv =
\frac{ B(\frac{3}{2}, \gamma + 1) }{c(1,\gamma)}
= \frac{1}{2\gamma+3}.
\]
Similarly,
\[
  V_4(1, \gamma) = \int_{-\infty}^\infty x^4 \kappa^\gamma(x) dx =
 \frac{2}{c(1,\gamma)} \int_{0}^1 v^2 (1 - v)^\gamma  \frac{1}{2 v^{1/2}} dv =
 \frac{ B(\frac{5}{2}, \gamma + 1) }{c(1,\gamma)} =
 \frac{3}{(2\gamma+3)( 2\gamma + 5)}.
\]

For dimensions $d>1$, write $V(d,\gamma)$ and $V_4(d,\gamma)$ as
a repeated integral and note that the innermost integral takes the form
\[
\int_{ \left\{ \frac{s^2}{1 - || x ||_{d-1}^2} \leq 1 \right\} }
   s^\alpha (1 - || x ||_{d-1}^2 - s^2)^\gamma ds
\]
for $\alpha \in \{ 2, 4 \}$.  By the symmetry and a change of parameters
$t = {s^2} / {(1 - || x ||_{d-1}^2)}$, it follows that
\[
  V(d,\gamma) = \frac{ B\left(\frac{3}{2}, \gamma+1 \right)}{c(d,\gamma)}
    c\left(d-1, \gamma+\frac{3}{2}\right)
\]
and
\[
  V_4(d,\gamma) = \frac{ B\left(\frac{5}{2}, \gamma+1 \right)}{c(d,\gamma)}
    c\left(d-1, \gamma+\frac{5}{2}\right)
\]
in accordance with the claim.

Finally for $d>1$, $  V_2(d, \gamma)$ can be written as
\[
\int_{ \{ ||x||_{d-1}^2 \leq 1 \} } \frac{x_{d-1}^2 }{c(d,\gamma)}
\left( 
\int_{ \left\{ \frac{s^2}{1 - || x ||_{d-1}^2} \leq 1 \right\} }
   s^2 (1 - || x ||_{d-1}^2 - s^2)^\gamma ds \right)
dx_1 \cdots dx_{d-1}.
\]
The inner integral is equal to
\(
( 1 - || x ||^2_{d-1})^{\gamma+3/2} B\left(\frac{3}{2}, \gamma + 1\right)
\)
so
\[
V_2(d,\gamma) = 
\frac{B\left(\frac{3}{2}, \gamma + 1\right) }{c(d, \gamma)}
c\left(d-1,\gamma+\frac{3}{2}\right) V\left(d-1, \gamma + \frac{3}{2}\right)
\]
in accordance with the claim.
\end{proof}

In the sequel, the following additional properties of the Beta kernels
will be needed.

\begin{res} \label{L:PI}
Consider the Beta kernels $\kappa^\gamma$ with $\gamma > 1$ defined in
equation (\ref{e:beta}). Then, for all $i\in \{ 1, \dots, d \}$,
\[
  \int_{\oR^d} u_i D_i\kappa^\gamma(u) du_1 \cdots du_d    = - 1,
\]
the integrals of second order products in $u \in \oR^d$ with respect to
$D_i \kappa^\gamma$ vanish and for distinct $i, j \in \{ 1, \dots, d \}$,
\begin{eqnarray*}
\int_{\oR^d} u_i u_j^2 D_i\kappa^\gamma(u) du_1 \cdots du_d
  & = & - V(d,\gamma) \\
\int_{\oR^d} u_i^3  D_i\kappa^\gamma(u) du_1 \cdots du_d 
  & = & -3 V(d,\gamma).
\end{eqnarray*}
The integrals of other third order products in $u \in \oR^d$ with respect to
$D_i \kappa^\gamma$ vanish. Finally the following identities hold for all
$i \neq j \in \{ 1, \dots, d \}$:
\[
  \int_{\oR^d} u_i u_j \sum_{k=1}^d u_k D_k \kappa^\gamma(u) du_1 \cdots du_d  =
 0
\]
and
\[
  \int_{\oR^d} u_i ^2\sum_{k=1}^d u_k D_k \kappa^\gamma(u) du_1 \cdots du_d  =
    -(d+2) V(d,\gamma).
\]
\end{res}

\begin{proof}{Lemma~\ref{L:PI}}
The proof relies on partial integrations, which involve evaluations
of
\(
    u_i^n u_j^m ( 1 - || u||^2 ) ^\gamma
\)
for
\(
|u_i| = ( 1 - || u_{(-i)} ||^2 )^{1/2}
\)
where $ u_{(-i)} = ( u_1, \dots, u_{i-1}, u_{i+1}, \dots, u_d)$.
These take the value zero, as $(1-||u||^2) = 0$.
Therefore
\[
\int_{\oR^d} u_i D_i\kappa^\gamma(u) du = - \int_{\oR^d} \kappa^\gamma(u) du = -1.
\]
Similarly
\[
  \int_{\oR^d} u_i u_j^2 D_i\kappa^\gamma(u) du = -  \int_{\oR^d} u_j^2 \kappa^\gamma(u) du =
  - V(d,\gamma)
\]
and
\[
  \int_{\oR^d} u_i^3 D_i\kappa^\gamma(u) du = - 3 \int_{\oR^d} u_i^2 \kappa^\gamma(u) du =
  - 3V(d, \gamma).
\]
Hence, for $i\neq j$, penultimate equation in the lemma holds.
To prove the last equation in the lemma, note that there
are $d-1$ contributions of $-V(d,\gamma)$ to the left-hand side 
and one of $-3V(d,\gamma)$. 
\end{proof}

\begin{res}\label{L:PIhigh}
Consider the Beta kernels $\kappa^\gamma$ with $\gamma > 2$ defined in
equation (\ref{e:beta}). Then, for all $i\neq j \in \{ 1, \dots, d \}$,
\[
  \int_{\oR^d} u_i u_j  \sum_{k=1}^d \sum_{l=1}^d u_k u_l D_{kl} \kappa^\gamma(u)
  du_1 \cdots du_d  =  0
\]
and
\[
 \int_{\oR^d} u_i^2 \sum_{k=1}^d \sum_{l=1}^d u_k u_l D_{kl} \kappa^\gamma(u)
 du_1 \cdots du_d  = (d+2) (d+3) V(d,\gamma).
\]
\end{res}

\begin{proof}{Lemma~\ref{L:PIhigh}}
Apply integration by parts and Lemma~\ref{L:PI} to obtain
that for all distinct $i, j, k, l$ in $\{ 1, \dots, d \}$,
\begin{eqnarray*}
 \int_{\oR^d} u_i^2 u_k u_l D_{kl} \kappa^\gamma(u) du_1 \cdots du_d & = & V(d,\gamma) \\
  \int_{\oR^d} u_k^3 u_l D_{kl} \kappa^\gamma(u)  du_1 \cdots du_d& = & 3V(d,\gamma) \\
 \int_{\oR^d} u_i^2 u_k^2 D_{kk} \kappa^\gamma(u) du_1 \cdots du_d & = & 2V(d,\gamma) \\
 \int_{\oR^d} u_k^4 D_{kk} \kappa^\gamma(u) du_1 \cdots du_d & = & 12V(d,\gamma).
\end{eqnarray*}
The evaluations of products in $u$ multiplied by $D_k\kappa^\gamma(u)$ are zero
since
\[
D_k \kappa^\gamma(u) \propto u_k ( 1 - || u ||^2 )^{\gamma - 1},
\]
which take the value zero when $||u|| = 1$ and $\gamma > 1$.
All other integrals of fourth order products in $u \in \oR^d$ with respect
to $D_{kl}\kappa^\gamma$ or $D_{kk}\kappa^\gamma$ vanish.

Consider the two equations to be proven. For $i\neq j$, all contributions to
the left-hand side of the first equation are zero. For $i=j$, there are
$(d-1)^2$ contributions with $k, l \not \in \{ i \}$,
of which $(d-1)(d-2)$ are of size $V(d,\gamma)$ for $k\neq l$ and
$d-1$ of size $2V(d,\gamma)$ for $k=l$.
To this are added $2(d-1)$ contributions $3V(d,\gamma)$ when exactly one
of $k, l$ is equal to $i$, and one contribution $12 V(d,\gamma)$ when
$i=k=l$. Adding them all up gives
\[
(d-1)(d-2) V(d,\gamma) + (d-1) 2 V(d,\gamma) + 2(d-1) 3V(d,\gamma) + 12 V(d,\gamma)
\]
and rearranging terms completes the proof.
\end{proof}

\begin{res} \label{L:Dgu}
For fixed $u\in \oR^d$, the function  $g_u: \oR \to \oR$ defined by
$g_u(v)  = v^{d+2} \kappa^\gamma(vu)$ is, for the Beta kernel $\kappa^\gamma$
with $\gamma > 4$, four times continuously differentiable. The first 
three derivatives are given by
\begin{eqnarray*}
  g^\prime_u(v)  & = & (d+2) v^{d+1} \kappa^\gamma(vu) +
  v^{d+2} D\kappa^\gamma(vu) u, \\
  g^{\prime\prime}_u(v) & = & (d+1) (d+2) v^{d} \kappa^\gamma(vu) +
          2 (d+2) v^{d+1} D\kappa^\gamma(vu) u +
          v^{d+2} D^2\kappa^\gamma(vu) (u, u), \\
  g^{\prime\prime\prime}_u(v) & = & d (d+1) (d+2) v^{d-1} \kappa^\gamma(vu) +
  3 (d+1) (d+2) v^{d}  D\kappa^\gamma(vu) u \\
  & + &
   3 (d+2) v^{d+1} D^2\kappa^\gamma(vu) (u, u)
 + v^{d+2} D^3\kappa^\gamma(vu) (u, u, u)
\end{eqnarray*}
and the fourth order derivative is
\[
  g^{(iv)}_u(v)  =  (d-1)d (d+1) (d+2) v^{d-2} \kappa^\gamma(vu) +
  4d(d+1)(d+2) v^{d-1}D\kappa^\gamma(vu) u +
\]
\[
      6 (d+1)(d+2)  v^d D^2\kappa^\gamma(vu) (u, u)
   +  4 ( d+2) v^{d+1}  D^3\kappa^\gamma(vu) (u, u, u)
   +  v^{d+2} D^4\kappa^\gamma(vu)(u, u, u, u).
\]
\end{res}

\begin{proof}{Lemma~\ref{L:Dgu}} 
For $\gamma > 4$, the function $\kappa^\gamma$ is four times continuously
differentiable. The expressions for the derivatives follow by straightforward 
calculation.
\end{proof}

\begin{res}\label{L:PI4-high}
Consider the Beta kernels $\kappa^\gamma$ with $\gamma > 4$ defined in
equation (\ref{e:beta}). Then, for all $i \in \{ 1, \dots, d \}$,
\[
   \int_{\oR^d} u_i^5 D_i \kappa^\gamma(u)du = -5 V_4(d, \gamma); \quad
   \int_{\oR^d} u_i^6 D_{ii} \kappa^\gamma(u)du = 30 V_4(d, \gamma); 
 \]
 \[
   \int_{\oR^d} u_i^7 D_{iii} \kappa^\gamma(u)du = -210 V_4(d, \gamma); \quad
   \int_{\oR^d} u_i^8 D_{iiii} \kappa^\gamma(u)du = 1680 V_4(d, \gamma).
 \]
\end{res}

\begin{proof}{Lemma~\ref{L:PI4-high}}
The proof relies on repeated integration by parts. The evaluations
of $u_i^m (1 - || u_{(-i)}||^2)^{\alpha}$ at $|u_i| = ( 1 - ||u_{(-i)}||^2)^{1/2}$
where $u_{(-i)} = (u_1, \dots, u_{i-1}, u_{i+1}, \dots, u_d)$ all take the
value zero for $0 < \alpha \leq \gamma$.
\end{proof}

\subsection{Proofs of propositions and theorems: non-adaptive case}
\label{S:proofs-nA}

\begin{proof}{Proposition~\ref{P:mse}}
Since $\widehat{\lambda(x_0)}$ is the average of $n$ independent random
variables $\widehat{\lambda(x_0; h, \Phi_i, W)}$, $i=1, \dots, n$,
\[
\EE \widehat{\lambda(x_0)} = \EE \widehat{\lambda(x_0; h, \Phi_1, W)}
\]
and
\[
  \var \widehat{\lambda(x_0)} =
  \frac{1}{n} \var \widehat{\lambda(x_0; h, \Phi_1, W)}.
\]
Therefore, by Lemma~\ref{L:moments},
\[
\EE \widehat{\lambda(x_0)} =
\frac{1}{h^d} \int_{b(x_0, h) \cap W} 
\kappa^\gamma\left( \frac{x_0 - u}{h} \right) \lambda(u) du
\]
and
\[
\var \widehat{\lambda(x_0)}   = 
\frac{1}{n h^{2d}}
\int_{b(x_0, h) \cap W} 
\int_{b(x_0, h) \cap W}
   \kappa^\gamma\left( \frac{x_0 - u}{h} \right)
   \kappa^\gamma\left( \frac{x_0 - v}{h} \right)
   ( g(u,v) - 1 ) \lambda(u) \lambda(v) du dv
\]
\[
 + 
\frac{1}{nh^{2d}}
\int_{b(x_0, h) \cap W} \kappa^\gamma\left( \frac{x_0 - u}{h} \right)^2
   \lambda(u) du.
\]
Since $\mse\widehat{\lambda(x_0)}$ is the sum of the squared bias
and the variance, the claim is seen to hold.
\end{proof}

\begin{proof}{of Theorem~\ref{P:biasVar}}
To prove 1.\ note that since $h_n$ goes to zero, $x_0 \in W$ and $W$ is open,
for $n$ large enough $b(x_0, h_n) \cap W$ is equal to $b(x_0, h_n)$. 
For such $n$, by a change of variables, the symmetry of the Beta kernels and 
the proof of Proposition~\ref{P:mse}, the bias is 
\begin{equation}\label{e:biasHn}
 \int_{b(0,1)} \kappa^\gamma(u) \left\{ \lambda(x_0+h_n u)- \lambda(x_0) \right\} du.
\end{equation}
The intensity $\lambda(x_0)$ can be brought under the integral since
$\kappa^\gamma$ is a probability density. 

Fix $u \in b(0,1)$.  As $x_0 + t h_n u \in W$ for all $0\leq t \leq 1$ and 
$\lambda$ is twice continuously differentiable on $W$, the term between curly 
brackets in the integrand may be expanded as a Taylor series (\ref{e:Taylor}) 
with $k=2$:
\[
\lambda(x_0+h_n u)- \lambda(x_0) = h_n D\lambda(x_0) u +
\frac{h_n^2}{2} D^2 \lambda(x_0 + \theta h_n u) (u, u)
\]
for some $0 < \theta = \theta(u) < 1$ that may depend on $u$. Write
\[
 D^2 \lambda(x_0 + \theta h_n u) (u, u) =
 D^2 \lambda(x_0 + \theta h_n u) (u, u) -  D^2 \lambda(x_0) (u, u) 
+
 D^2 \lambda(x_0) (u, u). 
\]
Now,
\[
\left| D^2 \lambda(x_0 + \theta h_n u) (u, u) -  D^2 \lambda(x_0) (u, u) \right|
= \left| 
\sum_{i=1}^d \sum_{j=1}^d u_i u_j ( \lambda_{ij}(x_0 + \theta h_n u)
    - \lambda_{ij}(x_0) ) 
\right| 
\]
is dominated by
\[
  \sum_{i=1}^d \sum_{j=1}^d \left| \lambda_{ij}(x_0 + \theta h_n u) 
    -  \lambda_{ij}(x_0) ) \right| 
\]
since $|u_i| \leq 1$. Since $n$ was chosen large enough for
$x_0 + \theta h_n u$ to lie in $W$, we may use the H\H{o}lder 
assumption to obtain the inequality
\[
  \left| D^2 \lambda(x_0 + \theta h_n u) (u, u) -  D^2 \lambda(x_0) (u, u) \right|
\leq
C \sum_{i=1}^d \sum_{j=1}^d || \theta h_n u ||^\alpha 
\leq d^2 C h_n^\alpha .
\]
The right hand side does not depend on the particular choice of
$u \in b(0,1)$ nor on $\theta(u) \in (0,1)$. 
In summary,
\[
\lambda(x_0 + h_n u)- \lambda(x_0) = h_n D\lambda(x_0) u + 
\frac{h_n^2}{2} D^2 \lambda(x_0)(u,u) + R(h_n, u)
\]
for a remainder term $R(h_n, u)$ that satisfies 
$| R(h_n, u) | \leq  C d^2 h_n^{2+\alpha} / 2$.

Returning to the bias (\ref{e:biasHn}), for large $n$,
\begin{eqnarray*}
\bias \widehat{ \lambda(x_0) }  & = &
h_n \int_{b(0,1)} \kappa^\gamma(u) D\lambda(x_0) u du 
 + 
 \frac{h_n^2}{2} \int_{b(0,1)} \kappa^\gamma(u) D^2 \lambda(x_0)(u,u) du
\\
& + &
 \int_{b(0,1)} \kappa^\gamma(u) R(h_n, u) du.
\end{eqnarray*}
By Lemma~\ref{L:Beta},
\[
h_n \int_{b(0,1)} \kappa^\gamma(u) D\lambda(x_0) u du =
h_n  \sum_{i=1}^d  D_i \lambda(x_0) \int_{b(0,1)}  u_i \kappa^\gamma(u) du = 0.
\]
Furthermore,
\begin{eqnarray*}
\frac{h_n^2}{2}
\int_{b(0,1)} \kappa^\gamma(u) D^2 \lambda(x_0) (u, u) du 
& = &
\frac{h_n^2}{2} 
\sum_{i=1}^d \sum_{j=1}^d \lambda_{ij}(x_0) 
\int_{b(0,1)}  u_i u_j \kappa^\gamma(u)  du
  \\
& = &
\frac{h_n^2}{2} 
\sum_{i=1}^d \lambda_{ii}(x_0) V(d,\gamma)
\end{eqnarray*}
because by Lemma~\ref{L:Beta}, the cross terms with $i\neq j$ are zero.
Finally, since $\kappa^\gamma$ is a probability density and $R(h_n,u)$ is
uniformly bounded in $u \in b(0,1)$,
\[
\left | \int_{b(0,1)} \kappa^\gamma(u) R(h_n, u) du \right| \leq 
 \frac{ C d^2}{2} h_n^{2 + \alpha} .
\]

\bigskip

To prove 2. note that, as for the bias, $n$ may be chosen large enough for the
ball $b(x_0, h_n)$ to fall entirely in $W$. For such $n$, by a change
of variables $u=(x-x_0)/h_n$ and the symmetry of the Beta kernels,
\[
\frac{1}{ n h_n^{2d} }
\int_{b(x_0, h_n)} \kappa^\gamma\left(\frac{x_0-x}{h_n}\right)^2 \lambda(x) dx =
\frac{1}{ n h_n^{d}}
\int_{b(0,1)} \kappa^\gamma(u)^2 \lambda(x_0 + h_n u) du.
\]
Fix $u \in b(0,1)$. As $x_0 + t h_n u \in W$ for all $0\leq t \leq 1$ and 
$\lambda$ is continuously differentiable on $W$, 
we may use the Taylor expansion (\ref{e:Taylor}) with $k=1$ to write
\[
  \lambda(x_0 + h_n u) = \lambda(x_0) + h_n D\lambda(x_0 + \theta h_n u) u
   = \lambda(x_0) + h_n \sum_{i=1}^d D_i \lambda(x_0 + \theta h_n u) u_i
\]
for some $0 < \theta = \theta(u) < 1$ that may depend on $u$.
Since the partial derivatives are continuous and hence bounded on 
closed balls contained in $W$, say by $\overline{ D\lambda }$, 
\begin{equation}\label{e:lambdaT1}
\lambda(x_0 + h_n u) = \lambda(x_0) + R(h_n, u)
\end{equation}
for a remainder term $R(h_n, u)$ that satisfies $|R(h_n, u)| \leq d h_n
\overline{ D\lambda }$ and consequently 
\[
\frac{1}{ n h_n^d }
\int_{b(0,1)} \kappa^\gamma(u)^2 \lambda(x_0 + h_n u) du
=
\frac{1}{ n h_n^d } \lambda(x_0) Q(d,\gamma)
+ \frac{1}{ n h_n^d }
\int_{b(0,1)} \kappa^\gamma(u)^2 R(h_n, u) du
\]
by Lemma~\ref{L:Beta}. The bound on the remainder term $R(h_n,u)$ implies that 
\[
\left| \frac{1}{ n h_n^d } \int_{b(0,1)} \kappa^\gamma(u)^2 R(h_n, u) du \right|
\leq  \frac{1}{ n h_n^d } \int_{b(0,1)} \kappa^\gamma(u)^2 | R(h_n, u) | du 
\leq \frac{d h_n \overline{ D\lambda } }{ n h_n^d} Q(d, \gamma) 
\]
so that 
\begin{equation}\label{e:varPoisHn}
  \frac{1}{ n h_n^{2d} }
\int_{b(x_0, h_n)} \kappa^\gamma\left(\frac{x_0-x}{h_n}\right)^2 \lambda(x) dx =
\frac{\lambda(x_0) Q(d,\gamma)}{ n h_n^d }  +
O\left( \frac{1}{nh_n^{d-1}} \right).
\end{equation}

We will now show that the contribution of the interaction structure
(through the pair correlation function) to the variance vanishes.
Choose $n$ so large that $b(x_0, h_n) \subseteq W$. Then, by a change of 
variables and the symmetry of the Beta kernels, the double integral in 
Proposition~\ref{P:mse} reduces to
\[
  \frac{1}{n} \int_{b(0,1)} \int_{b(0,1)}
 \kappa^\gamma(u) \kappa^\gamma(v)
( g( x_0 + h_nu, x_0 + h_nv) - 1 ) 
\lambda(x_0 + h_nu) \lambda(x_0 + h_nv)  du dv.
\]
Since the pair correlation is assumed to be bounded on $W$, say 
$g(\cdot, \cdot) \leq \overline g$, and $x_0 + h_nu \in W$ for all
$u\in b(0,1)$, the double integral can be bounded in absolute value by
\[
  \frac{ 1 + \overline g}{n} \left( \int_{b(0,1)} 
   \kappa^\gamma(u)  \lambda(x_0 + h_nu)  du \right)^2
=
  \frac{ 1 + \overline g}{n} \left( \int_{b(0,1)} 
   \kappa^\gamma(u) \left\{ 
       \lambda(x_0) + R(h_n,u) \right\} du \right)^2,
\]
cf.\ equation~(\ref{e:lambdaT1}). The integrand in the right
hand side is bounded in absolute value by
\(
   \kappa^\gamma(u) \left\{
     \lambda(x_0) +  d h_n \overline{D\lambda} 
   \right\}
\)
and therefore the interaction structure contributes $O(1/n)$ to the 
mean squared error. Upon adding (\ref{e:varPoisHn}),  
\[
  \var \widehat{\lambda(x_0)} = 
\frac{ \lambda(x_0) Q(d,\gamma) }{ n h_n^d } +
O\left(\frac{1}{nh_n^{d-1}}\right) + O\left(\frac{1}{n} \right).
\]
The last term is negligible with respect to the middle one, and the
proof is complete.
\end{proof}

\begin{proof}{Corollary~\ref{C:optimal}}
By Theorem~\ref{P:biasVar},
\[
\left( \bias \widehat \lambda(x_0) \right)^2 
 = 
h_n^4 \left(
  \frac{\sum_{i=1}^d \lambda_{ii}(x_0)}{2 (d + 2\gamma + 2) }
      \right)^2 
   +  2 h_n^2 R(h_n)
       \frac{\sum_{i=1}^d \lambda_{ii}(x_0)}{2 (d + 2\gamma + 2) }
+ R(h_n)^2
\]
for a remainder term $R(h_n)$ for which there exists a scalar $M$ such
that $|R(h_n)| \leq M h_n^{2 + \alpha}$ for large $n$. Hence
\[
  \left( \bias \widehat \lambda(x_0) \right)^2  =
  h_n^4 \left(
  \frac{\sum_{i=1}^d \lambda_{ii}(x_0)}{2 (d + 2\gamma + 2) }
      \right)^2
 + O(h_n^{4 + \alpha})
\]
and the claimed expression for the mean squared error follows.
Consequently, the asymptotic mean squared error takes the form
\[
  \alpha h_n^4 + \frac{\beta}{ n h_n^d}
\]
for some scalars $\alpha, \beta > 0$. Equating
the derivative 
with respect to $h_n$ to zero yields
\[
(h_n^*)^{3 + d + 1} = \frac{d \beta}{4 n \alpha}.
\]
The second derivative with respect to $h_n$,
\(
12 \alpha h_n^2 + d ( d+1) \beta n^{-1} h_n^{-d-2},
\)
is strictly positive, so $h_n^*$ is the unique minimum.
Plugging in the expressions for $\alpha$ and $\beta$ completes the proof.
\end{proof}

\begin{proof}{Proposition~\ref{P:bigO}}
Since $h_n \to 0$, $x_0\in W$ and $W$ is open, if $n$ is large enough then 
$b(x_0,h_n) \cap W = b(x_0, h_n)$. For such $n$, by Lemma~\ref{L:moments},
\[
\widehat{ \lambda(x_0) } - \EE \widehat{ \lambda(x_0) } =
\widehat{ \lambda(x_0) } - 
\frac{1}{h_n^d} \int_{b(x_0, h_n)}
\kappa^\gamma\left(\frac{x_0-x}{h_n}\right) \lambda(x) dx
=
\frac{1}{n} \sum_{i=1}^n Z_i
\]
can be written as an average of independent random variables
\[
  Z_i := \widehat{ \lambda(x_0; \Phi_i, h_n)} - 
\frac{1}{h_n^d} \int_{b(x_0, h_n)}
\kappa^\gamma\left(\frac{x_0-x}{h_n}\right) \lambda(x) dx
\]
with $\EE Z_i = 0$.
Furthermore, by Theorem~\ref{P:biasVar}, 
\[
  \var\left( \frac{1}{n} \sum_{i=1}^n Z_i \right)  = 
  \frac{\lambda(x_0) Q(d,\gamma)}{n h_n^d} + R(h_n)
\]
for a remainder term $R(h_n)$ satisfying $n h_n^{d-1} |R(h_n)| \leq M$
for some $M>0$ and large $n$. 
By Chebychev's inequality, for all $\epsilon > 0$,
\[
  \PP\left( \left| \frac{1}{n} \sum_{i=1}^n Z_i \right| \geq
    \epsilon^{-1/2}  \sqrt { \frac{\lambda(x_0) Q(d,\gamma)} { n h_n{d} } }
     \right)
     \leq \epsilon \frac{ n h_n^{d}}{ \lambda(x_0) Q(d,\gamma)}
     \left( \frac{\lambda(x_0) Q(d,\gamma)}{n h_n^d} + R(h_n) \right).
\]
The upper bound tends to $\epsilon$
as $n\to \infty$ so that
\[
  \frac{1}{n} \sum_{i=1}^n Z_i = O_P\left(
    \sqrt{ \frac{\lambda(x_0) Q(d,\gamma)}{  nh_n^{d} }}
\right).
\]
To finish the proof, add the bias expansion 1.\ in Theorem~\ref{P:biasVar}.
\end{proof}

\subsection{Proofs of propositions and theorems: adaptive case}
\label{S:proofs-A}

\begin{proof}{Theorem~\ref{P:biasVarA}}
To prove 1.\ note that since $h_n$ goes to zero, $x_0 \in W$,
$W$ is open and $\lambda$ is bounded away from zero, for $n$ large
enough
\[
  b(x_0, h_n / c(x)) \subseteq
  b(x_0, \lambda(x_0)^{1/2} h_n / \underline \lambda^{1/2}) \subset W
\]
for all $x\in W$.  For such $n$, by a change of variables, the symmetry of 
the Beta kernels and Lemma~\ref{L:momentsA}, the bias is equal to
\begin{eqnarray}
  \nonumber
  \bias \tilde \lambda(x_0) & = & \int_{\oR^d}
    \left[ \lambda(x_0) c(x_0+h_nu)^{d+2}
      \kappa^\gamma\left( u  c(x_0 + h_nu) \right)
  - \lambda(x_0) \kappa^\gamma(u) \right] du \\
  & = &
\label{e:biasHnA}
     \lambda(x_0) \int_{\oR^d} \left[ g_u(c(x_0+h_nu)) - g_u(1) \right] du
\end{eqnarray}
for the functions $g_u: \oR \to \oR$, $u\in\oR^d$, defined by
\[
g_u(v) = v^{d+2} \kappa^\gamma(v u).
\]
Note that the integral in (\ref{e:biasHnA}) is compactly supported, 
say on $K \subset \oR^d$, a property it inherits from the Beta kernel 
since $c$ is bounded away from zero.

Since we are after the coefficient of $h_n^2$ and, for $\gamma > 2$,
$\kappa^\gamma$ is twice continuously differentiable, we use a
Taylor expansion (\ref{e:Taylor}) with $k=2$. Thus, fix $u\in K$.
Then
\begin{eqnarray}
  \nonumber
  g_u(1+v) - g_u(1) & = & Dg_u(1)v + R(u, v) \\
  & = &
  (d+2) \kappa^\gamma( u) +
   \sum_{i=1}^d D_i \kappa^\gamma(u) u_i + R(u,v)
   \label{e:TaylorG}
\end{eqnarray}
where the remainder term is 
\[
R(u,v) =  \frac{v^2}{2} D^2 g_u(1 + \theta v)
\]
for some $0 < \theta = \theta(v) < 1$ that may depend on $v\in \oR$.
Moreover, $D^2 g_u(v)$ can be written as 
\[
 (d+1) (d+2) v^{d} \kappa^\gamma(v u) +
    2(d+2)  v^{d+1} \sum_{i=1}^d D_i \kappa^\gamma(v u) u_i
   +  v^{d+2} \sum_{i=1}^d \sum_{j=1}^d D_{ij} \kappa^\gamma(v u) u_i u_j.
\]
Recall that $g_u$ is evaluated at $v$ of the form $c(x_0+h_nu)$. Since
the function $c$ is bounded we may restrict ourselves to a compact interval 
$I$ for $v$ and on this interval $D^2g_u(v)$ is bounded as $\kappa^\gamma$ 
and its partial derivatives are bounded too. Moreover,
the bound can be chosen uniformly in $u$ over the compact set $K$.
In summary, there exists a constant $C>0$ such that $|R(u,v)|\leq
Cv^2$ for all $u\in K$ and $v\in I$.

We also need a Taylor expansion (\ref{e:Taylor}) with $k=2$ for the function $c$
around $x_0 \in \oR^d$: 
\begin{equation}
  c(x_0 + h_nu) - 1  =  h_n \sum_{i=1}^d D_ic(x_0) u_i + \tilde R_n(u) 
   =   h_n \sum_{i=1}^d \frac{D_i \lambda(x_0)}{2 \lambda(x_0)} u_i
        +  \tilde R_n(u)
\label{e:TaylorF}
\end{equation}
where the remainder term is
\[
  \tilde R_n(u) = 
  \frac{h_n^2}{2} D^2 c(x_0 + \theta h_n u) (u, u) =
  \frac{h_n^2}{2} \sum_{i=1}^d \sum_{j=1}^d D_{ij} c(x_0 + \theta h_n u) u_i u_j
\]
for some $0 < \theta = \theta(u) < 1$ that may depend on $u\in K$.
The second order partial derivatives are, for $i, j \in  \{ 1, \dots, d \}$,
\[
  D_{ij} c(u) = \frac{1}{2 \sqrt{\lambda(x_0)}}
  \frac{\lambda_{ij}(u) }{ \lambda(u)^{1/2} } 
  - \frac{1}{4 \sqrt{\lambda(x_0)}} \frac{ \lambda_i(u) \lambda_j(u)
    }{ \lambda(u)^{3/2} }
\]
where we use the notation $\lambda_i = D_i\lambda$. On the compact set $K$, 
the $|u_i|$ are bounded and so are the $|D_{ij}c(u)|$ since $\lambda$ is 
bounded away from zero and twice continuously differentiable. Hence there
exists a constant $\tilde C>0$ such that 
$|\tilde R_n(u)| \leq \tilde C h_n^2$ for all $u\in K$.

Our next step is to combine the Taylor series (\ref{e:TaylorG})
and (\ref{e:TaylorF}). Write $E_n(u) := c( x_0+h_nu) - 1$. For large $n$
the bias (\ref{e:biasHnA}) can then be written as
\begin{eqnarray}
  \nonumber
  \bias \tilde \lambda(x_0) & = & \lambda(x_0) \int_{\oR^d}
\left[  g_u( 1 + E_n(u) ) - g_u(1) \right] du \\
& = & \lambda(x_0) \int_{\oR^d}
  \left[
    E_n(u) Dg_u(1) + \frac{E_n(u)^2}{2} D^2 g_u(1+\eta(u) E_n(u))
  \right]
      du
      \label{e:biasAbramson}
\end{eqnarray}
for some $\eta(u)\in (0,1)$.

We will show that the first and second order terms vanish.
By (\ref{e:TaylorG})--(\ref{e:TaylorF}), the first order term is 
equal to $h_n \lambda(x_0)$ multiplied by
\[
\int_{\oR^d} Dg_u(1)  Dc(x_0) u du =
\sum_{i=1}^d D_i c(x_0) 
  \int_{\oR^d} u_i \left[
    (d+2) \kappa^\gamma(u) +  \sum_{j=1}^d D_j \kappa^\gamma(u) u_j 
  \right] du
\]
and vanishes because of Lemma~\ref{L:Beta} and Lemma~\ref{L:PI}.

Also by (\ref{e:TaylorG})--(\ref{e:TaylorF}), the 
second order term reads
\(
{h_n^2 \lambda(x_0)} I_n / 2
\)
where 
\[
I_n = \int_{\oR^d} \left[ Dg_u(1) D^2c(x_0+\theta(u) h_nu)(u,u)
    + D^2g_u( 1 + \eta(u) E_n(u) ) \left\{ D c(x_0) u 
      \right\}^2 \right] du
\]
for some $\theta(u)$ and $\eta(u)$ in $(0,1)$. Recall that $I_n$
is compactly supported and that the integrand is bounded. Therefore, by the
dominated convergence theorem, 
\begin{eqnarray*}
  \lim_{n\to\infty} I_n & = &
 \sum_{i=1}^d \sum_{j=1}^d
  D_{ij} c(x_0) \int_{\oR^d} u_i u_j \left[
    (d+2)\kappa^\gamma(u) + \sum_{k=1}^d D_k \kappa^\gamma(u) u_k \right] du 
  \\
 & + &  \sum_{i=1}^d \sum_{j=1}^d D_i c(x_0) D_j c(x_0) 
       \int_{\oR^d}  u_i u_j \times \\
  & \times &
  \left[
    (d+1) (d+2) \kappa^\gamma(u) +
    2 (d+2) \sum_{k=1}^d D_k \kappa^\gamma(u) u_k + \sum_{k=1}^d \sum_{l=1}^d
    D_{kl} \kappa^\gamma(u) u_k u_l 
    \right] du.
\end{eqnarray*}
The first double sum is zero because of Lemma~\ref{L:Beta} and Lemma~\ref{L:PI},
the second one because of Lemma~\ref{L:Beta}, Lemma~\ref{L:PI}
and Lemma~\ref{L:PIhigh}. By the bounds on the remainder terms $R$ and $\tilde R$,
all other terms in (\ref{e:biasAbramson}) are of the order $o(h_n^2)$
and the proof is complete.

\bigskip

To prove 2.\ note that, as for the bias, $n$ may be chosen so large that
\[
  b(x_0, h_n / c(x)) \subseteq
  b(x_0, \lambda(x_0)^{1/2} h_n / \underline \lambda^{1/2}) \subset W.
\]
For such $n$, by a change of variables
\(
  u =  (x - x_0) / h_n
\)
and the symmetry of the Beta kernels, 
\begin{eqnarray}
  \nonumber
  \int_W \frac{c(x)^{2d}}{nh_n^{2d}}
      \kappa^\gamma\left( \frac{x_0 - x}{h} c(x) \right)^2 \lambda(x) dx 
& = &
 \int_{\oR^d} \frac{c(x_0 + h_nu)^{2d}}{nh_n^d}
\kappa^\gamma\left(u c(x_0 + h_nu) \right)^2 \lambda(x_0 + h_n u) du \\
& = & 
  \frac{\lambda(x_0)}{nh_n^d} \int_{\oR^d} h_u \left( c(x_0 + h_nu) \right) du
\label{e:var1Abramson}
\end{eqnarray}
for the function $h_u: \oR \to \oR$, $u\in \oR^d$, defined by
\[
  h_u(v) = v^{2d+2} \kappa^\gamma(uv)^2.
\]
Note that the integral in(\ref{e:var1Abramson}) is compactly supported, 
say on $K\subset \oR^d$, a property it inherits from the Beta kernel since $c$ 
is bounded away from zero.

Fix $u \in K$. Then by a Taylor  expansion (\ref{e:Taylor}) with $k=1$
\[
  h_u(1+v)  =  h_u(1) + Dh_u(1 + \theta(v) v)v 
\]
for some $0<\theta(v)<1$, with
\[
  D h_u(v) = 2 (d+1) v^{2d+1} \kappa^\gamma(uv)^2 + 
  2 v^{2d+2} \kappa^\gamma(uv) \sum_{i=1}^d D_i \kappa^\gamma(uv) u_i.
\]
Recall that $h_u$ is evaluated at $v$ of the form $c(x_0+h_nu)$. Since
the function $c$ is bounded we may restrict ourselves to a compact
interval $I$ for $v$ and on this interval $Dh_u(v)$ is bounded as 
$\kappa^\gamma$ and its partial derivatives are bounded too. 
Moreover, the bound can be chosen uniformly in $u$ over the
compact set $K$. In summary, there exists a constant $H$ such that
$|Dh_u(v)| \leq H$ for all $u\in K$ and $v\in I$. Hence, with $E_n(u) =
c(x_0 + h_nu) - 1$ as before, (\ref{e:var1Abramson}) can be written as
\[
  \frac{\lambda(x_0)}{nh_n^d} \int_{\oR^d} \left[
    h_u(1) + E_n(u) D h_u( 1 + \theta(u) E_n(u))
  \right] du 
  = \frac{\lambda(x_0)}{nh_n^d} \int_{\oR^d} \kappa^\gamma(u)^2 du + R_n
\]
for a remainder term
\[
  R_n = 
\frac{\lambda(x_0)}{nh_n^d} \int_{\oR^d} E_n(u) D h_u( 1 + \theta(u)E_n(u)) du
\]
with $0 < \theta(u) < 1$. By (\ref{e:TaylorF}),
\[
| E_n(u) D h_u( 1 + \theta(u) E_n(u)) | \leq H |E_n(u) | \leq
h_n  H \left| \sum_{i=1}^d \frac{\lambda_i(x_0)}{2\lambda(x_0)} u_i \right|
+ H \left| \tilde R_n(u) \right|.
\]
As $u \in K$ and, for such $u$, $|\tilde R_n(u)| \leq \tilde C h_n^2$, 
\[
\frac{\lambda(x_0)}{nh_n^d} \int_{\oR^d} h_u \left( c(x_0+h_nu) \right) du  =
\frac{\lambda(x_0)}{nh_n^d} Q(d,\gamma) + O\left(\frac{1}{nh_n^{d-1}}\right).
\]

We will finally show that the contribution of the interaction structure
(through the pair correlation function) to the variance (\ref{e:varA})
vanishes. Again, choose $n$ so large that 
\[
  b(x_0, h_n / c(x)) \subseteq
  b(x_0, \lambda(x_0)^{1/2} h_n / \underline \lambda^{1/2}) \subset W.
\]
For such $n$, by a change of variables and the symmetry of the Beta kernels, 
and writing $\bar g$ for an upper bound to the pair correlation function, 
the integral in the last line in (\ref{e:varA}) can be bounded
in absolute value by
\[
  \frac{(1+\bar g) \lambda(x_0)^2}{n}
  \left( \int_{\oR^d} c(x_0 + h_nu)^{d+2}
    \kappa^\gamma( u c(x_0 + h_nu) ) du \right)^2 =
  O\left( \frac{1}{n} \right)
\]
since the integral is compactly supported and both $c$ and $\kappa^\gamma$ 
are bounded.
\end{proof}

\begin{proof}{{\bf of Theorem~\ref{P:biasVarHall}}} \\
As in the proof or Theorem~\ref{P:biasVarA}, the bias is given by
(\ref{e:biasHnA}) and the integral involved is supported on a 
compact set $K\subset \oR^d$. Since we are after the coefficient of 
$h_n^4$ and, for $\gamma > 5$, a Taylor expansions (\ref{e:Taylor}) 
with $k=5$ applies for both $c$ and $g_u$. For the former, $E_n(u) =
c(x_0 + h_n u) - 1$ is equal to 
\begin{equation}\label{e:Hallc}
 h_n Dc(x_0) u +
  \frac{1}{2} h_n^2 D^2 c(x_0) (u,u) + 
  \frac{1}{6} h_n^3 D^3 c(x_0) (u,u,u) + 
  \frac{1}{24} h_n^4 D^4 c(x_0) (u,u,u,u) 
\end{equation}
up to a remainder term
\[
  \tilde R_n(u) = \frac{1}{120} h_n^5 D^5c(x_0 + \theta h_nu)(u,u,u,u,u)
\]
for some $0<\theta=\theta(u)<1$ that may depend on $u\in K$. Since
$\lambda$ is bounded away from zero and five times continuously
differentiable, $| \tilde R_n(u)| \leq \tilde C h_n^5$ for all $u\in K$.
Similarly, for fixed $u\in K$,
\[
  g_u(1+v) - g_u(1) = v g_u^\prime(1)+
  \frac{1}{2} v^2 g_u^{\prime\prime}(1) +
  \frac{1}{6} v^3 g_u^{\prime \prime \prime}(1) +
  \frac{1}{24} v^4 g_u^{(iv)}(1) + R(u,v),
\]
where $R(u,v) = v^2 D^5g_u(1+ \theta v) / 120$ for some $\theta = \theta(v)$
in $(0,1)$ that may depend on $v\in\oR$. Recall that $g_u$ is evaluated at 
$v$ of the form $c(x_0+h_nu)$, $u\in K$. Since the function $c$ is bounded 
we may restrict ourselves to a compact interval $I$ for $v$ and on this 
interval $D^5 g_u(v)$ is bounded as $\kappa^\gamma$ and its partial 
derivatives up to fifth order are bounded too. Moreover, the bound can 
be chosen uniformly in $u$ over the compact set $K$.
In summary, $|R(u,v)|\leq C |v|^5$ for $u\in K$ and $v\in I$.

Next, plug the Taylor expansions into (\ref{e:biasHnA}). Then
\[
  \bias \tilde \lambda(x_0) = \lambda(x_0) \int_{\oR^d} \left[
    g_u(1 + E_n(u) ) - g_u(1) \right] du = \lambda(x_0) \int_{\oR^d} R(u, E_n(u)) du
\]
\begin{equation} \label{e:Hall}
  + \lambda(x_0) \int_{\oR^d} \left[ g_u^\prime(1) E_n(u) +
  \frac{1}{2} E_n(u)^2 g_u^{\prime\prime}(1) +
  \frac{1}{6} E_n(u)^3 g_u^{\prime \prime \prime}(1) +
  \frac{1}{24} E_n(u)^4 g_u^{(iv)}(1) \right] du.
\end{equation}
By Theorem~\ref{P:biasVarA}, the first and second order terms are zero.
We will show that the third order term $h_n^3 \lambda(x_0) I_{n,3}$ 
vanishes too. By (\ref{e:Hallc}),
\[
  I_{n,3} =  \int_{\oR^d}\left[
    \frac{1}{6} g_u^\prime(1) D^3 c(x_0) ( u, u, u) +
    \frac{1}{2} g_u^{\prime\prime}(1) Dc(x_0) u D^2c(x_0)(u,u) +
   \frac{1}{6} g_u^{\prime\prime\prime}(1) (D c(x_0) u)^3    
  \right] du.
\]
Lemma~\ref{L:Dgu} implies that the first term of $I_{n,3}$ is 
\[
  \frac{1}{6} \sum_{i=1}^d \sum_{j=1}^d \sum_{k=1}^d 
   D_{ijk} c(x_0) \int_{\oR^d} u_i u_j u_k \left[
    (d+2) \kappa^\gamma(u) + \sum_l u_l D_l \kappa^\gamma(u)
    \right] du,
\]
which vanishes by the symmetry properties of $\kappa^\gamma$,
Lemma~\ref{L:Beta} and Lemma~\ref{L:PI}. By Lemma~\ref{L:Dgu}, the second term
is a linear combination of integrals of the form
\[
 \int_{\oR^d} u_i u_j u_k
 \left[
    (d+1)(d+2) \kappa^\gamma(u) + 2(d+2) \sum_{l=1}^d u_l D_l\kappa^\gamma(u)
    + \sum_{l=1}^d \sum_{m=1}^d u_l u_m D_{lm} \kappa^\gamma(u)
    \right] du
\]
which vanish because of the symmetry properties of the Beta kernel
and integration by parts. Similar arguments apply to the third and last 
term of $I_{n,3}$, which by Lemma~\ref{L:Dgu} is a linear combination of 
integrals of the form
\[
  \int_{\oR^d} u_i u_j u_k
\{
    d(d+1)(d+2) \kappa^\gamma(u) +
    3 (d+1)(d+2)\sum_{l=1}^d  u_l D_l \kappa^\gamma(u) +
  \]
  \[
    3(d+2) \sum_{l=1}^d \sum_{m=1}^d u_l u_m D_{lm} \kappa^\gamma(u) +
    \sum_{l=1}^d \sum_{m=1}^d \sum_{n=1}^d u_l u_m u_n D_{lmn} \kappa^\gamma(u) 
    \} du.
  \]
  
The coefficient of $h_n^4$ in (\ref{e:Hall}) reads $\lambda(x_0)
\int A(u; x_0) du$ with $A$ as claimed, and does not vanish in general. 
Finally, by the bounds on the remainder terms $R$ and $\tilde R_n$, 
all other terms in (\ref{e:Hall}) are of the order $o(h_n^4)$ and the proof
is complete.
\end{proof}

\begin{proof}{Proposition~\ref{P:dim1}}
By Theorem~\ref{P:biasVarHall}, the coefficient of $h_n^4$ is
\(
\lambda(x_0) \int_{\oR^d} A(u; x_0) du,
\)
where
\begin{eqnarray*}
  \frac{ A(u; x_0) }{u^4 } & = &\frac{1}{24} g_u^\prime(1) c^{(iv)}(x_0) +
  \frac{1}{4} g_u^{\prime\prime\prime}(1) (c^\prime(x_0))^2 c^{\prime\prime}(x_0)
+  \frac{1}{24} g_u^{(iv)}(1) ( c^\prime(x_0))^4 \\
  & + & 
     \frac{1}{2} g_u^{\prime \prime}(1) \left[
     \frac{1}{3} c^\prime(x_0) c^{\prime \prime\prime}(x_0) +
     \frac{1}{4} (c^{\prime\prime}(x_0))^2
   \right] .
\end{eqnarray*}
Lemma~\ref{L:Dgu} and Lemma~\ref{L:PI4-high} can be used to derive
the following equations:
\begin{eqnarray*}
\int_{\oR^d} u^4 g_u^\prime(1) du & = & \int_{\oR^d} u^4 \left[
     3 \kappa^\gamma(u) + u D \kappa^\gamma(u) \right] du = -2 V_4(1,\gamma); \\
\int_{\oR^d} u^4 g_u^{\prime\prime}(1) du & = &  \int_{\oR^d} u^4 \left[
     6 \kappa^\gamma(u) + 6u D \kappa^\gamma(u) + u^2 D_{11} \kappa^\gamma(u)
     \right] du = 6 V_4(1,\gamma); \\
\int_{\oR^d} u^4 g_u^{\prime\prime\prime}(1) du & = &  \int_{\oR^d} u^4 \left[
    6 \kappa^\gamma(u) + 18 u D\kappa^\gamma(u) + 9 u^2 D_{11} \kappa^\gamma(u)
    + u^3 D_{111}  \kappa^\gamma(u) \right] du \\
   & = & - 24 V_4(1,\gamma); \\
\int_{\oR^d} u^4 g_u^{(iv)}(1) du & = &  \int_{\oR^d} u^4 \left[
   24 u D_1 \kappa^\gamma(u) + 36 u^2 D_{11} \kappa^\gamma(u) +
   12 u^3 D_{111}  \kappa^\gamma(u) + u^4 D_{1111}\kappa^\gamma(u)  \right] du \\
& = & 120 V_4(1, \gamma).
\end{eqnarray*}
Hence, upon a rearrangement of terms,
\begin{equation} \label{e:A}
  \int_{\oR^d} A(u; x_0) du =  \frac{-1}{12} c^{(iv)}(x_0)
  + c^{\prime\prime\prime}(x_0) c^\prime(x_0) +
  \frac{3}{4} ( c^{\prime\prime}(x_0) )^2
  - 6 c^{\prime\prime}(x_0) ( c^\prime(x_0) )^2 + 5 ( c^\prime(x_0) )^4.
\end{equation}

It remains to calculate and plug in expressions for the derivatives
of $c$ in terms of the underlying intensity function $\lambda$.
Now
 \begin{eqnarray*}
   c^\prime(x_0) & = &  \frac{\lambda^\prime(x_0)}{2\lambda(x_0)} \\
   c^{\prime\prime}(x_0) & = &   \frac{\lambda^{\prime\prime}(x_0)}{2\lambda(x_0)}
    - \frac{(\lambda^\prime(x_0))^2}{4 \lambda(x_0)^2} \\
   c^{\prime\prime\prime}(x_0) & = &
    \frac{\lambda^{\prime\prime\prime}(x_0)}{2\lambda(x_0)}
    -\frac{3\lambda^\prime(x_0)\lambda^{\prime\prime}(x_0)}{4 \lambda(x_0)^2}
   + \frac{3(\lambda^\prime(x_0))^3}{8\lambda(x_0)^3} \\
   c^{(iv)}(x_0) & = &
    \frac{\lambda^{(iv)}(x_0)}{2\lambda(x_0)}
    -\frac{4\lambda^\prime(x_0)\lambda^{\prime\prime\prime}(x_0) + 3(\lambda^{\prime\prime}(x_0))^2}{4 \lambda(x_0)^2}
   + \frac{9(\lambda^\prime(x_0))^2 \lambda^{\prime\prime}(x_0)}{4\lambda(x_0)^3} 
   - \frac{15(\lambda^\prime(x_0))^4}{16\lambda(x_0)^4} 
\end{eqnarray*}
can be plugged into (\ref{e:A}) to obtain
\[
  -\frac{1}{24} \frac{\lambda^{(iv)}(x_0)}{\lambda(x_0)} +
    \frac{1}{12}\frac{\lambda^\prime(x_0) \lambda^{\prime\prime\prime}(x_0)}{\lambda(x_0)^2}
    + \frac{1}{16}\frac{(\lambda^{\prime\prime}(x_0))^2}{\lambda(x_0)^2}
    -\frac{3}{16} \frac{\lambda^{\prime\prime}(x_0)(\lambda^\prime(x_0))^2}{\lambda(x_0)^3}
    + \frac{5}{64} \frac{(\lambda^\prime(x_0))^4}{\lambda(x_0)^4}
\]
\[
+  \frac{1}{4}\frac{\lambda^{\prime\prime\prime}(x_0)\lambda^\prime(x_0)}{\lambda(x_0)^2}
  - \frac{3}{8}\frac{ (\lambda^\prime(x_0))^2 \lambda^{\prime\prime}(x_0)}{\lambda(x_0)^3}
  + \frac{3}{16}\frac{(\lambda^\prime(x_0))^4}{\lambda(x_0)^4}
\]
\[
+  \frac{3}{16}\frac{(\lambda^{\prime\prime}(x_0))^2}{\lambda(x_0)^2}
-\frac{3}{16}\frac{\lambda^{\prime\prime}(x_0) (\lambda^\prime(x_0))^2}{\lambda(x_0)^3}
+ \frac{3}{64} \frac{(\lambda^\prime(x_0))^4}{\lambda(x_0)^4}
\]
\[
  -\frac{6}{8}\frac{\lambda^{\prime\prime}(x_0) (\lambda^\prime(x_0))^2}{\lambda(x_0)^3} +
  \frac{6}{16}\frac{(\lambda^\prime(x_0))^4}{\lambda(x_0)^4}
+ \frac{5}{16} \frac{(\lambda^\prime(x_0))^4}{\lambda(x_0)^4}
\]
and the claim follows upon a rearrangement of terms.
\end{proof}

\begin{proof}{Proposition~\ref{P:dim2}}
Theorem~\ref{P:biasVarHall} states that the coefficient of $h_n^4$ is
$\lambda(x_0) \int A(u; x_0) du$ with an explicit expression for $A(u; x_0)$.
The non-zero terms in this expression can be reduced by repeated partial
integration to a scalar multiple of either $V_4(2,\gamma)$ or $V_2(2,\gamma)$
as the integrals of other fourth order products in $u\in\oR^2$ with respect
to $\kappa^\gamma$ vanish by the symmetry properties of the Beta kernel.

The scalar multipliers can be calculated as in Lemma~\ref{L:PI4-high}:
for $i\neq j \in \{ 1, 2 \}$, integrals with respect to first order partial 
derivatives reduce to
\[
  \int_{\oR^d} u_i^4 u_j D_j\kappa^\gamma(u) du =  - V_4(2,\gamma); \quad
  \int_{\oR^d} u_i^3 u_j^2 D_i \kappa^\gamma(u) du  =  - 3 V_2(2,\gamma),
\]
integrals with respect to second order partial derivatives reduce to
\[
  \int_{\oR^d} u_i^5 u_j D_{ij} \kappa^\gamma(u) du  =  5 V_4(2,\gamma); \quad
  \int_{\oR^d} u_i^4 u_j^2 D_{jj} \kappa^\gamma(u) du  =  2 V_4(2,\gamma);
\]
\[
  \int_{\oR^d} u_i^4 u_j^2 D_{ii} \kappa^\gamma(u) du  =  12 V_2(2,\gamma); \quad
  \int_{\oR^d} u_i^3 u_j^3 D_{ij} \kappa^\gamma(u) du  =  9 V_2(2,\gamma),
\]
and integrals with respect to third order partial derivatives
are reduced as
\[
  \int_{\oR^d} u_i^6 u_j D_{iij} \kappa^\gamma(u) du =  -30 V_4(2,\gamma) 
\]
and
\[
  \int_{\oR^d} u_i^5 u_j^2 D_{ijj} \kappa^\gamma(u) du  =  -10 V_4(2,\gamma); \quad
  \int_{\oR^d} u_i^4 u_j^3 D_{jjj} \kappa^\gamma(u) du  =  -6V_4(2,\gamma),
\]
\[
   \int_{\oR^d} u_i^5 u_j^2 D_{iii} \kappa^\gamma(u) du  =  -60 V_2(2,\gamma); \quad
   \int_{\oR^d} u_i^4 u_j^3 D_{iij} \kappa^\gamma(u) du  =  -36 V_2(2,\gamma).
\]
Finally,
\[
  \int_{\oR^d} u_i^7 u_j D_{iiij} \kappa^\gamma(u) du  =  210 V_4(2,\gamma); \quad
  \int_{\oR^d} u_i^6 u_j^2 D_{iijj} \kappa^\gamma(u) du  =  60 V_4(2,\gamma);
\]
\[
  \int_{\oR^d} u_i^5 u_j^3 D_{ijjj} \kappa^\gamma(u) du =  30 V_4(2,\gamma); \quad
  \int_{\oR^d} u_i^4 u_j^4 D_{jjjj} \kappa^\gamma(u) du  =  24 V_4(2,\gamma),
\]
and
\[
  \int_{\oR^d} u_i^6 u_j^2 D_{iiii} \kappa^\gamma(u) du  =   360 V_2(2,\gamma); \quad 
  \int_{\oR^d} u_i^5 u_j^3 D_{iiij} \kappa^\gamma(u) du  =  180 V_2(2,\gamma);
\]
\[
  \int_{\oR^d} u_i^4 u_j^4 D_{iijj} \kappa^\gamma(u) du  =  144 V_2(2,\gamma) .
\]
Evaluation of the expression for $A(u; x_0)$ implies the claim by
elementary but tedious calculation. For example, the coefficient of
$D_{11}c(x_0) D_{22} c(x_0)$ arises from terms with these coefficients in
\[
  \frac{1}{8} \int_{\oR^d} (D^2 c(x_0)(u,u))^2 g_u^{\prime \prime}(1) du_1 du_2,
\]
which, by Lemma~\ref{L:Dgu}, is equal to
\[
  \frac{1}{8} \sum_{i=1}^2 \sum_{j=1}^2 \sum_{k=1}^2 \sum_{l=1}^2
  D_{ij}c(x_0) D_{kl} c(x_0) 
  \int_{\oR^d} u_i u_j u_k u_l \left[ 12 \kappa^\gamma(u) + 8 D \kappa^\gamma(u) +
        D^2\kappa^\gamma(u) \right] du_1 du_2.
\]
The desired coefficients occur when $i=j=1$ and $k=l=2$ or when
$i=j=2$ and $k=l=2$. Therefore
\[
\frac{D_{11}c(x_0) D_{22} c(x_0)}{4} \int_{\oR^d} u_1^2 u_2^2  \left[ 12 \kappa^\gamma(u) + 8 D \kappa^\gamma(u) +
        D^2\kappa^\gamma(u) \right] du_1 du_2,
\]
so the coefficient of $D_{11}c(x_0) D_{22} c(x_0)$ is equal to
\[
  \frac{1}{4} \left[ 12 V_2(2,\gamma) + 
    8 ( -3 - 3) V_2(2,\gamma) + ( 12 + 9 + 9 + 12 ) V_2(2,\gamma) \right]
  = \frac{6V_2(2,\gamma)}{4}.
\]
\end{proof}

\begin{proof}{Corollary~\ref{C:optimalHall}}
By Theorem~\ref{P:biasVarHall},
\[
\left( \bias \widehat \lambda(x_0) \right)^2 
=
  \lambda(x_0)^2 \left( \int_{\oR^d} A(u; x_0) du \right)^2 h_n^8 
   +  2 h_n^4 R(h_n) \lambda(x_0) \int_{\oR^d} A(u; x_0) du + R(h_n)^2
\]
for a remainder term $R(h_n)$ satisfying $R(h_n)/h_n^4 \to 0$
for large $n$. Hence
\[
  \left( \bias \widehat \lambda(x_0) \right)^2  =
  \lambda(x_0)^2 \left( \int_{\oR^d} A(u; x_0) du \right)^2 h_n^8 
 + o(h_n^{8})
\]
from which the claimed expression for the mean squared error follows.
Consequently, the asymptotic mean squared error takes the form
\[
  \alpha h_n^8 + \frac{\beta}{ n h_n^d}
\]
for some scalars $\alpha, \beta > 0$. Equating the derivative 
with respect to $h_n$ to zero yields
\[
(h_n^*)^{7 + d + 1} = \frac{d \beta}{8 n \alpha}.
\]
The second derivative with respect to $h_n$,
\(
  56 \alpha h_n^6 + d ( d+1) \beta n^{-1} h_n^{-d-2},
\)
is strictly positive, so $h_n^*$ is the unique minimum. Plugging in 
the expressions for $\alpha$ and $\beta$ completes the proof.
\end{proof}

\begin{proof}{Proposition~\ref{P:bigOHall}}
Since $h_n \to 0$, $x_0 \in W$ and $W$ is open, if $n$ is large enough the ball centred
at $x_0$ with radius $\lambda(x_0)^{1/2} h_n / \underline \lambda^{1/2}$
is contained in $W$. For such $n$, by Lemma~\ref{L:momentsA},
\[
\tilde \lambda(x_0) - \EE \tilde \lambda(x_0) = 
\tilde \lambda(x_0) - 
\frac{1}{h_n^d} \int_{\oR^d} c(x)^d \kappa^\gamma\left(\frac{x_0-x}{h_n} c(x) \right)
\lambda(x) dx
=
\frac{1}{n} \sum_{i=1}^n Z_i
\]
can be written as an average of independent random variables
\[
  Z_i :=  \tilde \lambda(x_0; \Phi_i, h_n) -
\frac{1}{h_n^d} \int_{\oR^d} c(x)^d \kappa^\gamma\left(\frac{x_0-x}{h_n} c(x) \right)
\lambda(x) dx
\]
with $\EE Z_i = 0$. Furthermore, by Theorem~\ref{P:biasVarHall} 
\[
  \var\left( \frac{1}{n} \sum_{i=1}^n Z_i \right)  = 
  \frac{\lambda(x_0) Q(d,\gamma)}{n h_n^d} + R(h_n)
\]
for a remainder term $R(h_n)$ satisfying $n h_n^{d-1} |R(h_n)| \leq M$
for some $M>0$ and large $n$. By Chebychev's inequality, for all $\epsilon > 0$,
\[
  \PP\left( \left| \frac{1}{n} \sum_{i=1}^n Z_i \right| \geq
    \epsilon^{-1/2}  \sqrt { \frac{\lambda(x_0) Q(d,\gamma)} { n h_n{d} } }
     \right)
     \leq \epsilon \frac{ n h_n^{d}}{ \lambda(x_0) Q(d,\gamma)}
     \left( \frac{\lambda(x_0) Q(d,\gamma)}{n h_n^d} + R(h_n) \right).
\]
The upper bound tends to $\epsilon$ as $n\to \infty$ so that
\[
  \frac{1}{n} \sum_{i=1}^n Z_i = O_P\left(
    \sqrt{ \frac{\lambda(x_0) Q(d,\gamma)}{  nh_n^{d} }}
\right).
\]
To finish the proof, add the bias expansion 1.\ in Theorem~\ref{P:biasVarHall}.
\end{proof}

\end{document}